\documentclass[11pt,reqno]{amsart}
\usepackage{amssymb,mathrsfs,graphicx}
\usepackage{ifthen}
\usepackage{colortbl}
\definecolor{black}{rgb}{0.0, 0.0, 0.0}
\definecolor{red}{rgb}{1.0, 0.5, 0.5}
\provideboolean{shownotes} 
\setboolean{shownotes}{true} 
\newcommand{\margnote}[1]{
\ifthenelse{\boolean{shownotes}}%
{\marginpar{\raggedright\tiny\texttt{#1}}}%
{}%
}
\newcommand{\hole}[1]{
\ifthenelse{\boolean{shownotes}}%
{\begin{center} \fbox{ \rule {.25cm}{0cm} \rule[-.1cm]{0cm}{.4cm}
\parbox{.85\textwidth}{\begin{center} \texttt{#1}\end{center}} \rule
{.25cm}{0cm}}\end{center}} {} }

\title[Compressible Vlasov-Navier-Stokes equations]{A revisit to the large-time behavior of the Vlasov/compressible Navier-Stokes equations}

\author[Choi]{Young-Pil Choi}
\address[Young-Pil Choi]{\newline Fakult\"at f\"ur Mathematik \newline
Technische Universit\"at M\"unchen, Boltzmannstra\ss e 3, 85748, Garching bei M\"unchen, Germany}
\email{ychoi@ma.tum.de}

\numberwithin{equation}{section}

\newtheorem{theorem}{Theorem}[section]
\newtheorem{lemma}{Lemma}[section]

\newtheorem{proposition}{Proposition}[section]
\newtheorem{remark}{Remark}[section]

\newcommand{\R}{\mathbb R}

\newcommand{\T}{\mathbb T}

\newcommand{\bq}{\begin{equation}}
\newcommand{\eq}{\end{equation}}
\newcommand{\e}{\varepsilon}
\newcommand{\lt}{\left}
\newcommand{\rt}{\right}

\newcommand{\pa}{\partial}
\newcommand{\mb}{\mathcal{B}}
\newcommand{\me}{\mathcal{E}}
\newcommand{\md}{\mathcal{D}}
\newcommand{\ml}{\mathcal{L}}
\def\charf {\mbox{{\text 1}\kern-.30em {\text l}}}

\begin{document}
\allowdisplaybreaks

\date{\today}

\subjclass[]{}


\begin{abstract}We establish the large-time behavior for the coupled kinetic-fluid equations. More precisely, we consider the Vlasov equation coupled to the compressible isentropic Navier-Stokes equations through a drag forcing term. For this system, the large-time behavior shows the exponential alignment between particles and fluid velocities as time evolves. This improves the previous result by Bae-Choi-Ha-Kang[Disc. Cont. Dyn. Sys. A, 34, (2014), 4419--4458] in which they considered the Vlasov/Navier-Stokes equations with nonlocal velocity alignment forces for particles. Employing a new Lyapunov functional measuring the fluctuations of momentum and mass from the averaged quantities, we refine assumptions for the large-time behavior of the solutions to that system.
\end{abstract}

\maketitle \centerline{\date}


%
%
%
%
\section{Introduction} This paper is concerned with the large-time behavior for the Vlasov equation coupled to the compressible isentropic Navier-Stokes equations through a drag force. More precisely, let $f(x,v,t)$ be the one-particle distribution function at $(x,v)\in \T^N \times \R^N$, $N \geq 3$, and let $\rho(x,t)$ and $u(x,t)$ be the density and the corresponding bulk velocity of the fluid equations at a domain $(x,t) \in \T^N \times \R_+$. Then the Vlasov/compressible Navier-Stokes equations can be described by
\begin{align}\label{main-eq}
\begin{aligned}
&\pa_t f + v \cdot \nabla_x f + \nabla_v \cdot ((u-v)f) = Q(f,f), \quad (x,v,t) \in \T^N \times \R^N \times \R_+,\\[2mm]
&\pa_t \rho + \nabla_x \cdot (\rho u) = 0,\cr
&\pa_t (\rho u) + \nabla_x \cdot (\rho u \otimes u) + \nabla_x p(\rho) + Lu = -\int_{\R^N} (u-v)f \,dv,
\end{aligned}
\end{align}
subject to the initial data:
\bq\label{ini-main-eq}
(f(x,v,t),\rho(x,t),u(x,t))|_{t = 0} = (f_0(x,v),\rho_0(x),u_0(x)).
\eq
Here $Q$ denotes the particle interaction operator, and the pressure law $p$ and Lam\'e operator $L$ are given by 
\begin{align*}
\begin{aligned}
&p(\rho) = \rho^\gamma, \quad \mbox{with} \quad \gamma > 1, \cr
&Lu = -\mu \Delta_x u - (\mu + \lambda)\nabla_x \nabla_x \cdot u, \quad \mbox{with} \quad \mu > 0, \quad 2\mu + \lambda > 0.
\end{aligned}
\end{align*}

In the last years, the interactions between particles and fluid have received a great deal of attentions due to a number of their applications such as biotechnology, medicine, aerosols, and sprays \cite{BBJM,BDM,Will}. We refer to \cite{orou,RM} for more physical backgrounds of the spray model. 

Some of the previous works on the existence theory for the coupled kinetic-fluid equations can be summarized as follows. The first result on the weak solutions to the Vlasov-Stokes equations is performed in \cite{Ham}, and later this work is extended to the Vlasov-Navier-Stokes equations in \cite{BCHK1,BDGM,CL,WY,Yu}. With a diffusion term in the Vlasov equation, i.e., Vlasov-Fokker-Planck equations, there are a number of literature on the global existence of weak and strong solutions for the interactions with homogeneous or inhomogeneous fluids \cite{BCHK2,BD,CCK,CDM,CKL1,CKL2,DL,MV}. For those systems, we can easily find the equilibrium solutions, and this enables to show the large-time behavior of solutions converging to them. More specifically, classical solutions near Maxwellians converging asymptotically to them  are constructed in \cite{CDM,CKL2,DL}. On the other hand, without the diffusion term, it is not clear to find the nontrivial equilibria, and thus the previous arguments for the large-time behavior of solutions can not be applied. Recently, the author and his collaborators developed a new argument for that in \cite{BCHK2,BCHK,CCK,CK1}. In particular, in \cite{BCHK}, the existence of strong solutions and large-time behavior are first established for the system \eqref{main-eq} with nonlocal velocity alignment forces for particles.

In this paper, we revisit to the estimate of large-time behavior of solutions to the system \eqref{main-eq}. By employing a different Lyapunov functional from the one proposed in \cite{BCHK}, we refine assumptions on the solutions and show that the particles will be aligned with the fluid velocity exponentially fast as time evolves. Our strategy can also be applied for the system \eqref{main-eq} with nonlocal/local velocity alignment forces discussed in \cite{BCHK,CCK}(see Remarks \ref{rmk2} and \ref{rmk3} below for details) and two-phase fluid models \cite{Choi, CK2}. In particular, in \cite{CK2}, the a priori asymptotic behavior estimate plays an important role in constructing the global-in-time classical solutions.

Throughout this paper, we use several simplified notations. For a function $f(x,v)$, $\|f\|_{L^p}$ denotes the usual $L^p(\T^N \times \R^N)$-norm, and if $g$ is a function of $x$, the usual $L^p(\T^N)$-norm is denoted by $\|g\|_{L^p}$, unless otherwise specified. We also denote by $C$ a generic positive constant depending only on the norms of the data, but independent of time $t$, and drop $x$-dependence of differential operators $\nabla_x$, that is, $\nabla f := \nabla_x f$ and $\Delta f := \Delta_x f$.

In order to present our main theorem, we first define several averaged quantities which will be used for the estimates of large-time behavior of solutions later. We set $m(x,t) := (\rho u)(x,t)$, 
\[
v_c(t) := \frac{\int_{\T^N \times \R^N} vf\,dxdv}{\int_{\T^N \times \R^N} f\,dxdv}, \quad  m_c(t) := \frac{\int_{\T^N} \rho u\,dx}{\int_{\T^N} \rho \,dx},\quad f_c(t) := \int_{\T^N \times \R^N} f\,dxdv,
\]
and
\[
\rho_c(t):= \int_{\T^N} \rho\,dx.
\]
For the particle interaction operator, we assume that $Q$ satisfies
\bq\label{ass_q}
\int_{\R^N} Q(f,f)\,dv=\int_{\R^N} Q(f,f) v_i\,dv = 0 \quad \mbox{and} \quad \int_{\T^N \times \R^N} Q(f,f) |v|^2\,dxdv \leq 0,
\eq
for $i=1,\cdots, N$.
\begin{theorem}\label{thm_main} Let $(f,\rho,u)$ be any global classical solutions to the system \eqref{main-eq} satisfying
\begin{align*}
\begin{aligned}
&(i)\,\,\,\,\,\, \|\rho_f\|_{L^\infty(\R_+;L^{N/2}(\T^N))} < \infty \quad \mbox{where} \quad \rho_f(x,t) := \int_{\R^N} f(x,v,t)\,dv,\cr
&(ii)\,\,\,\,\rho(x,t) \in [0, \bar\rho] \quad \mbox{for all} \quad (x,t) \in \T^N \times \R_+ \quad \mbox{and} \quad \rho_c(0) > 0,\cr
&(iii)\,\, u \in L^\infty(\T^N \times \R_+) \quad \mbox{and} \quad E_0>0 \mbox{ is small enough},
\end{aligned}
\end{align*}
where $E_0$ is an initial total energy given in Remark \ref{en_rmk1}.
Then we have
\[
\ml(f(t),\rho(t),u(t)) \leq C\ml(f_0,\rho_0,u_0)e^{-Ct} \quad t \geq 0,
\]
where $C$ is a positive constant independent of $t$, and $\ml$ is given by
\bq\label{lyap_eq}
\ml(f,\rho,u) := \int_{\T^N \times \R^N} |v - v_c|^2 f\,dxdv + \int_{\T^N} \rho|u - m_c|^2dx + \int_{\T^N} (\rho - \rho_c)^2dx + |v_c - m_c|^2.
\eq
\end{theorem}
\begin{remark}\label{rmk1} The smallness condition for the initial total energy $E_0$ is required to have the proper dissipation rate of the density, find the details of it in the proof of Lemma \ref{lem_cri}.
\end{remark}
\begin{remark}\label{rmk2} In \cite{CCK} and \cite{BCHK1, BCHK2,BCHK}, local and nonlocal velocity alignment forces
\[
Q_1(f,f):= - \nabla_v \cdot((u_f-v)f) \quad \mbox{and} \quad  Q_2(f,f) := -\nabla_v \cdot (F(f)f),
\]
are taken into account as the particle interactions, respectively. Here $F(f)$ and $u_f$ represent nonlocal velocity alignment forces and local particle velocity, respectively, given by
\[
F(f)(x,v,t) := \int_{\T^N \times \R^N}\psi(x-y)(w-v)f(y,w,t)\,dydw,
\] 
where $\psi$ is a positive symmetric function and
\[
u_f(x,t) := \frac{\int_{\R^N} vf(x,v,t)\,dv}{\int_{\R^N} f(x,v,t)\,dv}.
\]
We can easily find that those particles interaction operators $Q_i(f,f),i=1,2$ satisfy our assumption \eqref{ass_q}. It is also clear that Boltzmann's collision operator satisfies \eqref{ass_q} due to the fundamental properties of conserving mass, momentum, and energy. We refer to \cite{Math} for the local-in-time existence of smooth solutions for kinetic-fluid equations with collision operators.
\end{remark}
\begin{remark}\label{rmk3} For the system \eqref{main-eq} with nonlocal alignment forces in velocities, i.e., the system \eqref{main-eq} with $Q = Q_2$, a global existence of unique strong solution is studied in \cite{BCHK}. More precisely, for a given $T \in (0,\infty)$, assume that the data $(f_0,\rho_0,u_0)$ satisfy the normalization, compactly supported in velocity, smallness, and non-vacuum conditions:
\begin{align*}
\begin{aligned}
&\mbox{supp}_v (f_0(x,\cdot)) \mbox{ is bounded for all } x \in \T^3,\cr 
& \|f_0\|_{W^{1,\infty}} + \|u_0\|_{H^2} < \e_1, \quad \|\rho_0\|_{H^2} < \e_1^{7/16}, \quad \inf_{x \in \T^3}\rho_0(x) \geq \e_1^{1/2},
\end{aligned}
\end{align*}
where a small positive constant $\e_1$ has a order $\e_1e^{CT} = \mathcal{O}(1)$ and $C$ is a sufficiently large constant independent of $T > 0$. Then there exists a unique strong solution $(f,\rho,u)$ to the system \eqref{main-eq}-\eqref{ini-main-eq} such that
\begin{align*}
\begin{aligned}
& f \in W^{1,\infty}(\T^3 \times \R^3 \times (0,T)), \quad (\rho,\rho_t) \in L^\infty(0,T;H^2(\T^3)) \times L^\infty(0,T;H^1(\T^3)),\cr
& u \in L^\infty(0,T;H^2(\T^3)) \cap L^2(0,T;H^3(\T^3)) \cap H^1(0,T;H^1(\T^3)).
\end{aligned}
\end{align*}
Note that the strategy in \cite{BCHK} for the existence of strong solutions can directly be applied for the system \eqref{main-eq} in three dimensions.
\end{remark}
\begin{remark}\label{rmk4}  For the estimate of large-time behavior of solutions, the following functional $\ml_{ex}(f,\rho,u)$ is considered in \cite{BCHK}: 
\[
\ml_{ex}(f,\rho,u) := \int_{\T^3 \times \R^3} |v - v_c|^2 f\,dxdv + \int_{\T^3} \rho|u - u_c|^2dx + \int_{\T^3} (\rho - \rho_c)^2dx + |v_c - u_c|^2,
\]
where $u_c(t) := \int_{\T^3} u\, dx$, and it is showed that $\ml_{ex}(f,\rho,u)$ exponentially decays to zero under the smallness assumptions on the solutions such as $\|(\rho,u,\rho_f)\|_{L^\infty} \ll 1$. Compared to the work in \cite{BCHK}, we proposed the different Lyapunov functional $\ml(f,\rho,u)$ appeared in \eqref{lyap_eq}, and this enables to have more refined assumptions on the solutions for the large-time behavior estimate. In particular, we do not require the smallness of solutions $\rho,u$, and $\rho_f$ in $L^\infty(\T^N)$. We only need the small total initial energy. We also notice that
\[
\ml(f,\rho,u) \leq C\ml_{ex}(f,\rho,u) \quad \mbox{for some positive constant} \quad C > 0.
\]
\end{remark}
\begin{remark}\label{rmk5} It follows from conservation of masses and total momentum(see Lemma \ref{lem_energy}) that 
\[
v_c(t) - m_c(t) = (f_c(0) + 1)v_c(t) - \frac{1}{\rho_c(0)}\lt(\int_{\T^N \times \R^N} vf_0(x,v)\,dxdv - \int_{\T^N} m_0(x)\,dx\rt).
\]
This yields that 
\[
v_c(t),\,m_c(t) \to \frac{1}{\rho_c(0)\lt(f_c(0) + 1\rt)}\lt(\int_{\T^N \times \R^N} vf_0(x,v)\,dxdv - \int_{\T^N} m_0(x)\,dx\rt),
\]
as $t \to \infty$.
\end{remark}
\begin{remark}\label{rmk6} Theorem \ref{thm_main} implies
\[
\lim_{t \to \infty} d_{BL}\lt(f(x,v,t), \rho_f(x,t)\delta_{v(t)}(v)\rt) = 0,
\]
where $d_{BL}$ denotes the bounded Lipschitz distance. Indeed, if we set 
\[
\mathcal{S} := \lt\{ \varphi:\T^N \times \R^N \to \R~:~ \|\varphi\|_{L^\infty} \leq 1 \mbox{ and } \|\varphi\|_{Lip}:=\sup_{(x,v) \neq (y,w)}\frac{|\varphi(x,v) - \varphi(y,w)|}{|(x,v)-(y,w)|} \leq 1 \rt\},
\]
then for $\varphi \in \mathcal{S}$ we find
$$\begin{aligned}
&\lt|\int_{\T^N \times \R^N} \varphi(x,v)f(x,v,t)\,dxdv - \int_{\T^N \times \R^N} \varphi(x,v)\rho_f(x,t)\delta_{v(t)}(v)\,dxdv\rt|\cr
&\quad = \lt|\int_{\T^N \times \R^N} \varphi(x,v)f(x,v,t)\,dxdv - \int_{\T^N \times \R^N} \varphi(x,v(t))f(x,v,t)\,dxdv\rt|\cr
&\quad \leq \int_{\T^N \times \R^N} |\varphi(x,v) - \varphi(x,v(t))|f(x,v,t)\,dxdv\cr
&\quad \leq \lt(\int_{\T^N \times \R^N}|v-v(t)|^2f(x,v,t)\,dxdv\rt)^{1/2} f_c(0)^{1/2} \to 0 \quad \mbox{as} \quad t \to \infty.
\end{aligned}$$
This and together with Remark \ref{rmk5} further yields
\[
\lim_{t \to \infty} d_{BL}\lt(f(x,v,t), \rho_f(x,t)\delta_{v^\infty}(v)\rt) = 0,
\]
where $v^\infty \in \R^N$ is given by
\[
v^\infty := \frac{1}{\rho_c(0)\lt(f_c(0) + 1\rt)}\lt(\int_{\T^N \times \R^N} vf_0(x,v)\,dxdv - \int_{\T^N} m_0(x)\,dx\rt).
\]
\end{remark}
The rest of the paper is organized as follows. In Section \ref{sec_pre}, we present basic energy estimates for the system \eqref{main-eq} and a type of Bogovskii's result in the periodic domain which will be used later. Section 3 is devoted to the details of the proof for Theorem \ref{thm_main}.
%
%
%
%
\section{Preliminaries}\label{sec_pre}
In this section, we provide several useful lemmas which will be used for the proof of Theorem \ref{thm_main}.

We first show the energy estimates in the lemma below. Since its proof is by now standard, we omit it here.
\begin{lemma}\label{lem_energy} Let $(f,\rho,u)$ be any global classical solutions to the system \eqref{main-eq}-\eqref{ini-main-eq}. Then we have
\begin{align*}
\begin{aligned}
&(i) \,\,\,\,\,\, \frac{d}{dt} \int_{\T^N \times \R^N} f\,dxdv = \frac{d}{dt}\int_{\T^N} \rho\,dx = 0.\cr
&(ii) \,\,\,\, \frac{d}{dt}\lt( \int_{\T^N \times \R^N} v f\,dxdv + \int_{\T^N} \rho u\,dx\rt) = 0.\cr
&(iii)\,\,\frac12\frac{d}{dt}\lt( \int_{\T^N \times \R^N} |v|^2 f\,dxdv + \int_{\T^N} \rho |u|^2 dx + \frac{2}{\gamma - 1}\int_{\T^N} \rho^\gamma dx\rt) \cr
&\qquad + \mu\int_{\T^N} |\nabla u|^2 dx + (\mu + \lambda)\int_{\T^N} |\nabla \cdot u|^2 dx + \int_{\T^N \times \R^N} |u-v|^2 f\,dxdv\cr
&\qquad \quad = \frac12\int_{\T^N \times \R^N}|v|^2 Q(f,f)\,dxdv.
\end{aligned}
\end{align*}
\end{lemma}
\begin{remark}\label{en_rmk1}
Set 
\[
E(t):= \int_{\T^N \times \R^N} |v|^2 f\,dxdv + \int_{\T^N} \rho |u|^2 dx + \frac{2}{\gamma - 1}\int_{\T^N} \rho^\gamma dx,
\]
then we obtain
\[
E(t) \leq E(0) =:E_0 \quad \mbox{for} \quad t \geq 0,
\]
due to Lemma \ref{lem_energy} (iii).
\end{remark}
\begin{remark}\label{en_rmk2} It follows from Lemma \ref{lem_energy} (i) that
\[
\int_{\T^N \times \R^N} f\,dxdv = \int_{\T^N \times \R^N} f_0\,dxdv \quad \mbox{and} \quad \int_{\T^N} \rho\,dx = \int_{\T^N} \rho_0\,dx \quad \mbox{for} \quad t \geq 0,
\]
i.e., $f_c(t) = f_c(0)$ and $\rho_c(t) = \rho_c(0)$ for all $t \geq 0$. For notational simplicity, we set $f_c := f_c(0)$ and $\rho_c := \rho_c(0)$.
\end{remark}

We next recall the elliptic regularity estimate for Poisson's equation. 

\begin{lemma}\label{lem_poi} \cite[Lemma 7.9]{Majda} Let $f$ be a periodic function with zero average, i.e., $\int_{\T^N} f\,dx = 0$. Then the unique solution $\phi$ up to an additive constant of the Poisson equation 
\[
-\Delta \phi = f,
\]
satisfies the elliptic regularity estimate
\[
\|\phi\|_{H^{s+1}} \leq C\|f\|_{H^{s-1}}, \quad \mbox{for some positive constant } C >0.
\]
\end{lemma}
Then as a direct consequence of Lemma \ref{lem_poi}, we have a  type of Bogovskii's result in the periodic domain. We refer to \cite{Bogo,Gal} for the general bounded domain.

\begin{proposition}\label{prop_bogo}Given any $f \in L^2_{\#}(\T^N) := \Big\{ f \in L^2(\T^N) | \int_{\T^N} f dx = 0
\Big\}$, the following stationary transport equation
\bq\label{tran_eq}
\nabla \cdot \nu = f, \quad \nabla \times \nu = 0, \quad \mbox{and} \quad \int_{\T^N} \nu\,dx = 0,
\eq
admits a solution operator $\mb: f \mapsto \nu$ satisfying the following properties:
\begin{enumerate}
\item $\nu = \mb[f]$ is a solution to the problem \eqref{tran_eq} and a linear operator from $L^2_{\#}(\T^N)$ into $H^1(\T^N)$, i.e.,
\[
\|\mathcal{B}[f] \|_{H^1(\T^N)} \leq C\| f \|_{L^{2}(\T^N)}.
\]

\item If a function $f \in H^1(\T^N)$ can be written in the form
$f = \nabla \cdot g$ with $g \in [H^1(\T^N)]^N$,
    then
    \[
    \| \mathcal{B}[f] \|_{L^{2}(\T^N)} \leq C\|g\|_{L^{2}(\T^N)}.
    \]
\end{enumerate}
\end{proposition}
%
%
%
%
\section{Proof of Theorem \ref{thm_main}}

Before we deal with the Lyapunov function $\ml(f,\rho,u)$, we first provide some elementary estimates on the averaged quantities $m_c$ and $v_c$.
\begin{lemma}\label{lem_el} The followings hold.
\begin{align*}
\begin{aligned}
&(i) \,\,\,\,\,\, |m_c(t)|^2  \leq E_0/\rho_c \quad \mbox{and} \quad |m_c'(t)|^2 \leq f_c/\rho_c\int_{\T^N \times \R^N} |u-v|^2 f\,dxdv.\cr
&(ii) \,\,\,\, v_c'(t) = -(\rho_c/f_c) \,m_c'(t).
\end{aligned}
\end{align*}
where $'$ denotes the time-derivative, i.e., $g'(t) := \frac{d}{dt}g(t)$.
\end{lemma}
\begin{proof} (i) Using H\"older's inequality together with Lemma \ref{lem_energy} yields
\begin{align*}
\begin{aligned}
|m_c(t)| &\leq \frac{\int_{\T^N} \rho|u|dx}{\int_{\T^N} \rho\,dx} \leq \frac{1}{\rho_c}\lt(\int_{\T^N} \rho|u|^2 dx\rt)^\frac12\lt( \int_{\T^N} \rho\,dx\rt)^\frac12 \leq \rho_c^{-\frac12}E_0^\frac12,\cr
|m_c'(t)| & \leq \frac{1}{\rho_c}\lt|\frac{d}{dt}\int_{\T^N}\rho u\,dx\rt| \leq \frac{1}{\rho_c}\int_{\T^N} |u- v|f\,dxdv \cr
&\leq \frac{1}{\rho_c}\lt(\int_{\T^N \times \R^N} |u-v|^2 f\,dxdv\rt)^\frac12\lt(\int_{\T^N \times \R^N} f\,dxdv\rt)^\frac12\cr
&\leq (f_c/\rho_c)^\frac12\lt(\int_{\T^N \times \R^N} |u-v|^2 f\,dxdv\rt)^\frac12.
\end{aligned}
\end{align*} 
(ii) It follows from Lemma \ref{lem_energy} (ii) that
\[
f_c v_c'(t) + \rho_c m_c'(t) = 0, \quad \mbox{i.e.,} \quad v_c'(t) = \frac{\rho_c}{-f_c} m_c'(t).
\]
\end{proof}

We set temporal interacting energy functional $\mathcal{E}(f,\rho,u)$ and its corresponding dissipation $\mathcal{D}(f,\rho,u)$ as follows.
\begin{align}\label{t_energy}
\begin{aligned}
\mathcal{E}(f,\rho,u)&:= \int_{\T^N \times \R^N} |v - v_c|^2 f\,dxdv + \int_{\T^N} \rho|u - m_c|^2 dx + \frac{2}{\gamma-1}\int_{\T^N} \rho^\gamma dx \cr
&\quad + \frac{f_c}{2(f_c + \rho_c)}|m_c - v_c|^2,\cr
\mathcal{D}(f,\rho,u)&:= \mu\int_{\T^N} |\nabla u|^2 dx + (\mu + \lambda)\int_{\T^N} |\nabla \cdot u|^2 dx + \int_{\T^N \times \R^N} |u-v|^2 f\,dxdv.
\end{aligned}
\end{align}
\begin{remark}Our strategy for the estimate of large-time behavior heavily relies on that the system \eqref{main-eq} conserve the total momentum, see Lemma \ref{lem_energy} (ii). Thus, apart from the requirement of higher regularities of solutions appeared in Theorem \ref{thm_main}, we can not employ the above Lyapunov functional $\me$ and the dissipation $\md$ with weak solutions to the system \eqref{main-eq}. For weak solutions, the energy estimate in Lemma \ref{lem_energy} (iii) implies
\[
\lim_{t \to \infty} \int_t^{t+1}\int_{\T^N}\rho_f|u_f - u|^2\,dxds = \lim_{t \to \infty} \int_t^{t+1}\int_{\T^N \times \R^N}|u_f - v|^2f\,dxdvds = 0,
\]
since 
\[
\int_{\R^N} |u - v|^2 f\,dv = \rho_f |u_f - u|^2 + \int_{\R^N} |u_f - v|^2 f\,dv.
\]

\end{remark}
In the lemma below, we show that the temporal interacting energy functional is not increasing in time.
\begin{lemma}Consider the energy functional $\mathcal{E}$ and the dissipation $\mathcal{D}$ given in \eqref{t_energy}. Then we have
\[
\frac12\frac{d}{dt}\mathcal{E}(f,\rho,u) + \mathcal{D}(f,\rho,u) \leq 0.
\]
\end{lemma}
\begin{proof} It follows from \eqref{ass_q} that
\[
\int_{\T^N \times \R^N} | v - v_c|^2 Q(f,f)\,dxdv \leq 0.
\]
Then by definition of the averaged quantity $v_c$, we easily obtain 
\[
\frac12\frac{d}{dt}\int_{\T^N \times \R^N} |v - v_c|^2 f\,dxdv \leq \int_{\T^N \times \R^N} (v - v_c) \cdot (u-v)f\,dxdv.
\]
We also find 
\begin{align*}
\begin{aligned}
&\frac12\frac{d}{dt}\int_{\T^N} \rho|u - m_c|^2 dx + \mu\int_{\T^N} |\nabla u|^2 dx + (\mu + \lambda)\int_{\T^N} |\nabla \cdot u|^2 dx \cr
&\qquad = \int_{\T^N} (\nabla \cdot u)p(\rho)\,dx - \int_{\T^N \times \R^N} (u - m_c)\cdot (u-v) f \,dxdv.
\end{aligned}
\end{align*}
This yields 
\begin{align}\label{est_3_1}
\begin{aligned}
&\frac12\frac{d}{dt}\lt( \int_{\T^N \times \R^N} |v - v_c|^2 f\,dxdv + \int_{\T^N} \rho|u - m_c|^2 dx + \frac{2}{\gamma - 1}\int_{\T^N} \rho^\gamma dx\rt) \cr
&\quad + \mu \int_{\T^N} |\nabla u|^2 dx + (\mu + \lambda)\int_{\T^N} |\nabla \cdot u|^2 dx + \int_{\T^N \times \R^N} |u-v|^2 f\,dxdv\cr
&\qquad \leq (m_c - v_c)\cdot \int_{\T^N \times \R^N} (u-v)f\,dxdv,
\end{aligned}
\end{align}
where we used
\[
\frac{1}{\gamma - 1}\frac{d}{dt}\int_{\T^N} \rho^\gamma dx = -\int_{\T^N} (\nabla \cdot u)\,p(\rho)\,dx.
\]
On the other hand, we use Lemma \ref{lem_el} (ii) to get
\begin{align}\label{est_3_2}
\begin{aligned}
\frac{d}{dt}|m_c - v_c|^2 &= 2(m_c - v_c)\cdot (m_c' - v_c')\cr
&=2\lt(1 + \frac{\rho_c}{f_c} \rt)m_c' \cdot(m_c - v_c) \cr
&= -2\lt(\frac{f_c + \rho_c}{f_c} \rt)(m_c - v_c)\cdot \int_{\T^N \times \R^N}(u-v)f\,dxdv
\end{aligned}
\end{align}
Finally, we combine the above estimates \eqref{est_3_1} and \eqref{est_3_2} to complete the proof.
\end{proof}
Next we recall the quantitative estimate on the pressure $p(\rho)$.
\begin{lemma}\label{lem_pre}
\emph{\cite{FZZ}} (i) Let $r_0, \bar{r} > 0$ and $\gamma > 1$ be
given constants, and set
\[
f(r;r_0) := r\int_{r_0}^{r} \frac{h^{\gamma} - r_0^{\gamma}}{h^2} \,dh,
\]
for $r \in [0,\bar{r}]$. Then, there exist positive constants
$C_1$ and $C_2$ such that
\[
C_1(r_0, \bar{r})(r - r_0)^2 \leq f(r;r_0) \leq C_2(r_0, \bar{r})(r - r_0)^2 \quad \mbox{for all } r \in [0,\bar{r}].
\]
(ii) It holds 
\[
\frac{1}{\gamma-1}\frac{d}{dt}\int_{\T^N} \rho^\gamma dx = \frac{d}{dt}\int_{\T^N} f(\rho;\rho_c) \,dx.
\]
\end{lemma}
We reset our temporal interacting energy functional $\mathcal{E}$ as follows.
\begin{align*}
\begin{aligned}
\mathcal{E}(f,\rho,u)&:= \int_{\T^N \times \R^N} |v - v_c|^2 f\,dxdv + \int_{\T^N} \rho|u - m_c|^2 dx + 2\int_{\T^N}\rho \int_{\rho_c}^\rho \frac{h^\gamma - \rho_c^\gamma}{h^2}\,dh dx \cr
&\quad + \frac{f_c}{2(f_c + \rho_c)}|m_c - v_c|^2.
\end{aligned}
\end{align*}
Then it is clear from Lemma \ref{lem_pre} (ii) that 
\[
\frac12\frac{d}{dt}\mathcal{E}(f,\rho,u) + \mathcal{D}(f,\rho,u) \leq 0.
\]
Furthermore, it is obvious to get $\me(f,\rho,u) \approx \ml(f,\rho,u)$ in the sense that there exists a positive constant $C > 0$ such that
\bq\label{equiv}
\frac1C \ml(f(t),\rho(t),u(t)) \leq \me(f(t),\rho(t),u(t)) \leq C \ml (f(t),\rho(t),u(t)), \quad \mbox{for all} \quad t \geq 0,
\eq
due to Lemma \ref{lem_pre} (i). 

We now estimate the relation between dissipation $\md$ and Lyapunov functional $\ml$. We remark that if it is possible to get the following relation:
\[
\ml(f,\rho,u) \leq C\md(f,\rho,u) \quad \mbox{for some} \quad C > 0,
\]
then this directly concludes our desired exponential decay estimate of the functional $\ml$. Unfortunately, this does not hold due to the pressure term. To be more precise, we can not extract the desired dissipation rate of the density in $\mathcal{D}$. On the other hand, if we consider the pressureless viscous fluid \cite{Pere} or inhomogeneous fluid \cite{CK1, WY} instead of the compressible Navier-Stokes equations, we can show that the exponential alignment between particles and fluid velocities by using the Lyapunov functional $\ml$ and the energy functional $\me$ without the terms related to the pressure. In fact, the lemma below shows the Lyapunov functional $\ml$ is bounded by sum of the dissipation $\md$ and that term related to the pressure from above.
\begin{lemma}\label{lem_cri0} There exists a positive constant $C > 0$ such that
\[
\ml(f(t),\rho(t),u(t)) \leq C\md(f(t),\rho(t),u(t)) + \int_{\T^N} (\rho - \rho_c)^2\,dx\quad \mbox{for all} \quad t \geq 0,
\]
where $C$ is a positive constant depending on $f_c, \rho_c, \bar \rho$, and $\|\rho_f\|_{L^\infty(\R_+;L^{N/2}(\T^N))}$.

\end{lemma}
\begin{proof} By adding and subtracting, we find
\begin{align*}
\begin{aligned}
&\int_{\T^N \times \R^N} |u-v|^2 f \,dxdv\cr
&\qquad = \int_{\T^N \times \R^N} |u-m_c|^2 f\,dxdv + f_c |m_c - v_c|^2 + \int_{\T^N \times \R^N} |v - v_c|^2 f\,dxdv\cr
&\qquad \quad + 2\int_{\T^N \times \R^N} (u - m_c) \cdot (m_c - v_c)f\,dxdv + 2\int_{\T^N \times \R^N} (u - m_c) \cdot (v_c - v)f\,dxdv,\cr
&\qquad \geq -3\int_{\T^N \times \R^N} |u-m_c|^2 f\,dxdv + \frac{f_c}{2}|m_c - v_c|^2 + \frac12\int_{\T^N \times \R^N} |v - v_c|^2 f\,dxdv,
\end{aligned}
\end{align*}
where we used the Cauchy-Schwarz inequality and 
\[
\int_{\T^N \times \R^N} (m_c - v_c)\cdot (v_c - v) f\,dxdv = 0.
\]
This implies
\[
\frac{f_c}{2}|m_c - v_c|^2 + \frac12\int_{\T^N \times \R^N} |v - v_c|^2 f\,dxdv \leq \int_{\T^N \times \R^N} |u-v|^2 f \,dxdv + 3\int_{\T^N \times \R^N} |u-m_c|^2 f\,dxdv.
\]
We now claim that 
\[
\int_{\T^N \times \R^N} |u-m_c|^2 f\,dxdv \leq C\int_{\T^N}|\nabla u|^2 dx,
\]
for some positive constant  $C > 0$. For the proof of claim, we first divide the term in the left hand side of the above inequality into two parts:
\begin{align}\label{est_3_3}
\begin{aligned}
\int_{\T^N \times \R^N} |u-m_c|^2 f\,dxdv &\leq 2\int_{\T^N \times \R^N} |u-u_c|^2 f\,dxdv + 2f_c|u_c - m_c|^2\cr
&=: I_1 + I_2,
\end{aligned}
\end{align}
where $I_i,i=1,2$ are estimated as 
\begin{align}\label{est_3_4}
\begin{aligned}
I_1 & = 2\int_{\T^N} |u - u_c|^2 \rho_f\,dx \leq 2\lt(\int_{\T^N} |u - u_c|^{\frac{2N}{N-2}} dx \rt)^{\frac{N-2}{N}}\lt( \int_{\T^N} \rho_f^{\frac{N}{2}} dx\rt)^{\frac{2}{N}}\cr
&\leq C\|u-u_c\|_{L^{\frac{2N}{N-2}}}^2\|\rho_f\|_{L^{\frac{N}{2}}} \leq C\|\nabla u\|_{L^2}^2,\cr
I_2 & = 2f_c \lt| \frac{1}{\rho_c} \int_{\T^N} \rho u_c\,dx - \frac{1}{\rho_c} \int_{\T^N} \rho u\,dx \rt|^2 = 2f_c \lt| \frac{1}{\rho_c} \int_{\T^N} \rho (u_c - u)\,dx\rt|^2\cr
&\leq \frac{2f_c}{\rho_c^2}\lt( \int_{\T^N} \rho |u - u_c|^2 dx\rt) \lt(\int_{\T^N} \rho\,dx \rt)\leq \frac{2f_c \bar\rho}{\rho_c}\int_{\T^N} | u - u_c|^2 dx \leq C\|\nabla u\|_{L^2}^2,
\end{aligned}
\end{align}
where $\rho_f = \int_{\R^N} f\,dv$, and we used 
\[
\|u - u_c\|_{L^{\frac{2N}{N-2}}} \leq C\|u - u_c\|_{H^1} \leq C\|\nabla u\|_{L^2}.
\]
Combining \eqref{est_3_3} and \eqref{est_3_4}, we get
\[
\int_{\T^N \times \R^N} |u - m_c|^2 f\,dxdv \leq C\int_{\T^N} |\nabla u|^2 dx,
\]
where $C$ is a positive constant depending on $f_c, \,\rho_c, \,\bar \rho$, and $\|\rho_f\|_{L^\infty(\R_+;L^{N/2}(\T^N))}$. Thus, we obtain
\[
\frac{f_c}{2}|m_c - v_c|^2 + \frac12\int_{\T^N \times \R^N} |v - v_c|^2 f\,dxdv \leq \int_{\T^N \times \R^N} |u-v|^2 f \,dxdv + C\int_{\T^N} |\nabla u|^2 dx,
\]
and this deduces
\[
\ml(f,\rho,u) \leq C\lt( \md(f,\rho,u) +  \int_{\T^N} \rho |u - m_c|^2 \,dx\rt) + \int_{\T^N} (\rho - \rho_c)^2\,dx.
\]
Then we again use the similar argument as the above to find
\[
\int_{\T^N} \rho |u - m_c|^2 \,dx \leq C\bar\rho \int_{\T^N} |\nabla u|^2 dx.
\]
This concludes 
\[
\ml(f(t),\rho(t),u(t)) \leq C \md(f(t),\rho(t),u(t)) + \int_{\T^N} (\rho - \rho_c)^2\,dx, \quad t \geq 0,
\]
where $C$ is a positive constant depending on $f_c,\, \rho_c, \,\bar \rho$, and $\|\rho_f\|_{L^\infty(\R_+;L^{N/2}(\T^N))}$.
\end{proof}
Lemma \ref{lem_cri0} implies that we only need to focus on obtaining the proper dissipation rate of $\|\rho - \rho_c\|_{L^2}^2$. For this, we use the periodic version of Bogovskii's argument stated in Proposition \ref{prop_bogo}. We define perturbed interacting energy function $\mathcal{E}^\sigma(f,\rho,u)$ and dissipation $\mathcal{D}^\sigma(f,\rho,u)$ as follows.
\begin{align*}
\begin{aligned}
\mathcal{E}^\sigma(f,\rho,u) &:= \mathcal{E}(f,\rho,u) - 2\sigma\int_{\T^N} \rho(u-m_c)\cdot \mb[\rho - \rho_c]\,dx\cr
\mathcal{D}^\sigma(f,\rho,u) &:= \mathcal{D}(f,\rho,u) + \sigma\int_{\T^N} \rho u \cdot \nabla((u - m_c)\cdot \mb[\rho- \rho_c])dx \cr
&\quad - \sigma\int_{\T^N} \rho(u \cdot \nabla u) \cdot \mb[\rho- \rho_c]\,dx + \sigma\int_{\T^N} (\rho - \rho_c)\lt(\rho^\gamma - \rho_c^\gamma \rt)dx\cr
&\quad -\sigma\mu\int_{\T^N} \nabla u : \nabla \mb[\rho- \rho_c]\,dx + \sigma\int_{\T^N \times \R^N} (u-v) \cdot \mb[\rho- \rho_c] f\,dxdv\cr
&\quad - \sigma(\mu + \lambda)\int_{\T^N} (\nabla \cdot u) (\rho- \rho_c) \,dx - \sigma m_c' \cdot \int_{\T^N} \rho\mb[\rho- \rho_c]\,dx\cr
&\quad - \sigma\int_{\T^N}\rho(u - m_c)\cdot\mb[\nabla \cdot (\rho u)]\,dx\cr
&=:\sum_{i=1}^9 J_i.
\end{aligned}
\end{align*}
It is clear to get 
\[
\frac12\frac{d}{dt}\mathcal{E}^\sigma(f,\rho,u) + \mathcal{D}^\sigma(f,\rho,u) \leq 0.
\]
We notice that the perturbed interacting energy functional $\me^\sigma(f,\rho,u)$ satisfies $\me^\sigma(f,\rho,u) \approx \ml (f,\rho,u)$ for $\sigma > 0$ small enough, i.e., there exists a positive constant $C > 0$ such that 
\bq\label{equiv2}
\frac1C \ml(f,\rho,u) \leq \me^\sigma(f,\rho,u) \leq C \ml (f,\rho,u), 
\eq
due to 
\[
2\sigma\int_{\T^N} \rho(u-m_c)\cdot \mb[\rho - \rho_c]\,dx \leq C\bar\rho\sigma\int_{\T^N} \rho|u - m_c|^2 dx + C\sigma \int_{\T^N} (\rho - \rho_c)^2dx,
\]
and the relation \eqref{equiv}.

We next show that the perturbed dissipation $\mathcal{D}^\sigma$ dominates the dissipation rate $\mathcal{D}$ and $\|\rho - \rho_c\|_{L^2}^2$ under smallness assumptions on $\sigma$ and $E_0$.

\begin{lemma}\label{lem_cri} There exists a positive constant $C > 0$ independent of $t$ such that 
\[
\int_{\T^N} |\nabla u|^2 dx + \int_{\T^N} |\nabla \cdot u|^2 dx + \int_{\T^N \times \R^N} |u-v|^2 f\,dxdv + \int_{\T^N} (\rho - \rho_c)^2dx \leq C\mathcal{D}^\sigma (f,\rho,u),
\]
for sufficiently small $\sigma, \,E_0 > 0$.
\end{lemma}
\begin{proof}

We first estimate of $J_2$ and $J_3$ as follows.
\begin{align*}
\begin{aligned}
J_2 &= - \sigma\int_{\T^N} \nabla \cdot (\rho u) \lt((u-m_c)\cdot \mb[\rho - \rho_c]\rt)dx,\cr
J_3 &= - \sigma\int_{\T^N} \lt(\nabla \cdot (\rho u\otimes u)\rt)\cdot \mb[\rho- \rho_c]\,dx + \sigma\int_{\T^N} \lt(\nabla \cdot (\rho u) \,u \rt)\cdot \mb[\rho- \rho_c]\,dx.
\end{aligned}
\end{align*}
This yields 
\[
J_2 + J_3 = \sigma\int_{\T^N} (\rho u \otimes u) : \nabla \mb[\rho- \rho_c]\,dx + \sigma\int_{\T^N} \lt(\nabla \cdot (\rho u)\,m_c\rt) \cdot \mb[\rho- \rho_c]\,dx.
\]
Adding and subtracting, we deduce
\begin{align*}
\begin{aligned}
&\sigma\int_{\T^N} \rho(u \otimes u) : \nabla \mb[\rho- \rho_c]\,dx\cr
&\quad = \sigma\lt(\int_{\T^N} \rho(u - m_c)\otimes u : \nabla \mb[\rho- \rho_c]\,dx + \int_{\T^N} \rho m_c \otimes (u - m_c): \nabla \mb[\rho- \rho_c]\,dx\rt)\cr
&\qquad + \sigma\int_{\T^N}\rho m_c \otimes m_c : \nabla \mb[\rho- \rho_c]\,dx\cr
&\quad =: \sum_{i=1}^3 J_{2,3}^i,
\end{aligned}
\end{align*}
where $J_{2,3}^i,i=1,2,3$ are estimated by
\begin{align*}
\begin{aligned}
J_{2,3}^1 &\leq \sigma^{1/2}\bar \rho \|u\|_{L^\infty(\T^N \times \R_+)}\int_{\T^N} \rho |u - m_c|^2 dx + C \sigma^{3/2}\int_{\T^N} (\rho- \rho_c)^2 dx,\cr
J_{2,3}^2 &\leq \sigma^{1/2}\bar \rho E_0 \rho_c^{-1}\int_{\T^N} \rho |u - m_c|^2 dx + C \sigma^{3/2}\int_{\T^N} (\rho- \rho_c)^2 dx,\cr
J_{2,3}^3 &= \sigma\int_{\T^N} (\rho-\rho_c) m_c \otimes m_c : \nabla \mb[\rho- \rho_c]\,dx \leq C\sigma E_0 \rho_c^{-1} \int_{\T^N} (\rho- \rho_c)^2 dx,
\end{aligned}
\end{align*}
due to Lemma \ref{lem_el}.
Thus, we obtain
\begin{align*}
\begin{aligned}
&\sigma\int_{\T^N} \rho (u \otimes u) : \nabla \mb[\rho - \rho_c]\,dx \cr
& \,\, \leq C \sigma\lt( E_0 \rho_c^{-1} + \sigma^{1/2} \rt)\int_{\T^N} (\rho- \rho_c)^2 dx + \sigma^{1/2}\bar \rho \lt(\|u\|_{L^\infty(\T^N \times \R_+)} + E_0\rho_c^{-1}\rt)\int_{\T^N} \rho |u - m_c|^2 dx.
\end{aligned}
\end{align*}
Similarly, we find
\begin{align*}
\begin{aligned}
&\sigma\int_{\T^N} \lt(\nabla \cdot (\rho u)m_c\rt) \cdot \mb[\rho- \rho_c]\,dx\cr
&\quad = - \sigma\int_{\T^N} \rho u \cdot \lt(m_c \cdot \nabla \mb[\rho- \rho_c]\rt)dx\cr
&\quad = - \sigma\int_{\T^N} \rho (u - m_c)\cdot\lt( m_c \cdot \nabla \mb[\rho- \rho_c]\rt)dx + \sigma\int_{\T^N} \rho m_c \cdot \lt( m_c \cdot \nabla \mb[\rho- \rho_c]\rt)dx\cr
&\quad \leq \sigma^{1/2}\bar \rho E_0 \rho_c^{-1} \int_{\T^N} \rho |u - m_c|^2 dx + C\sigma\lt( E_0 \rho_c^{-1}+ \sigma^{1/2}\rt)\int_{\T^N} (\rho- \rho_c)^2dx.
\end{aligned}
\end{align*}
This deduces
\begin{align*}
\begin{aligned}
J_2 + J_3 &\leq C \sigma\lt( E_0 \rho_c^{-1} + \sigma^{1/2} \rt)\int_{\T^N} (\rho- \rho_c)^2 dx \cr
&\quad + \sigma^{1/2}\bar \rho \lt(\|u\|_{L^\infty(\T^N \times \R_+)} + E_0\rho_c^{-1}\rt)\int_{\T^N} \rho |u - m_c|^2 dx.
\end{aligned}
\end{align*}
We next estimate $J_i,i=4,\cdots,8$ as
\begin{align*}
\begin{aligned}
J_4 &\geq C\sigma\int_{\T^N} (\rho- \rho_c)^2 dx,\cr
J_5 &\leq \frac\mu2\int_{\T^N} |\nabla u|^2 dx + C\sigma^2 \mu \int_{\T^N} (\rho- \rho_c)^2dx,\cr
J_6 &\leq \sigma\lt(\int_{\T^N \times \R^N} |u-v|^2 f\,dxdv\rt)^\frac12\lt( \int_{\T^N} |\mb[\rho- \rho_c]|^2\rho_f\,dx\rt)^\frac12\cr
&\leq \frac12\int_{\T^N \times \R^N} |u-v|^2 f\,dxdv + C\sigma^2\|\rho_f\|_{L^\infty(\R_+;L^{N/2}(\T^N))} \int_{\T^N} (\rho- \rho_c)^2dx,\cr
J_7 &\leq \frac{\mu + \lambda}{2}\int_{\T^N} |\nabla \cdot u|^2 dx + C\sigma^2(\mu + \lambda)\int_{\T^N} (\rho- \rho_c)^2 dx,\cr
J_8 &\leq \bar \rho\sigma\int_{\T^N} |m_c'||\mb[\rho- \rho_c]|\,dx \cr
&\leq \sigma^{1/2}\bar\rho f_c \rho_c^{-1}\int_{\T^N \times \R^N} |u-v|^2 f\,dxdv + C\sigma^{3/2}\int_{\T^N} (\rho- \rho_c)^2 dx,
\end{aligned}
\end{align*}
where we used for the estimates of $J_6$ and $J_8$ that
\begin{align*}
\begin{aligned}
\int_{\T^N} |\mb[\rho-\rho_c]|^2 \rho_f\,dx &\leq \lt(\int_{\T^N} |\mb[\rho - \rho_c]|^{\frac{2N}{N-2}}dx\rt)^\frac{N-2}{N}\|\rho_f\|_{L^{N/2}} \cr
&\leq \|\mb[\rho- \rho_c]\|_{H^1}^2\|\rho_f\|_{L^\infty(\R_+;L^{N/2}(\T^N))}\cr
&\leq C\|\rho_f\|_{L^\infty(\R_+;L^{N/2}(\T^N))}\int_{\T^N} (\rho- \rho_c)^2 dx,\cr
\end{aligned}
\end{align*}
and
\[
|m_c'|^2 \leq f_c \rho_c^{-1} \int_{\T^N \times \R^N} |u-v|^2 f\,dxdv,
\]
respectively. Finally, by Proposition \ref{prop_bogo}, we get
\begin{align*}
\begin{aligned}
J_9 &= -\sigma\int_{\T^N} \rho (u - m_c)\cdot \mb[\nabla \cdot (\rho(u-m_c))]\,dx  - \sigma\int_{\T^N} \rho(u - m_c)\cdot \mb[\nabla \cdot ((\rho- \rho_c) m_c)]\,dx\cr
&\leq C\sigma\bar\rho\int_{\T^N} \rho|u - m_c|^2 dx + C\sigma\bar\rho^\frac12\lt( \int_{\T^N} \rho|u - m_c|^2 dx\rt)^\frac12\lt(\int_{\T^N} (\rho- \rho_c)^2 |m_c|^2dx\rt)^\frac12\cr
&\leq C(\sigma + \sigma^{1/2})\bar \rho\int_{\T^N}\rho|u - m_c|^2 dx + C\sigma^{3/2}E_0 \rho_c^{-1}\int_{\T^N} (\rho- \rho_c)^2 dx.
\end{aligned}
\end{align*}
Summing up the above estimates, we obtain
\begin{align*}
\begin{aligned}
\mathcal{D}^\sigma(f,\rho,u) &\geq \frac{\mu}{2}\int_{\T^N} |\nabla u|^2 dx + \frac{(\mu + \lambda)}{2}\int_{\T^N} |\nabla \cdot u|^2 dx \cr
&\quad + \lt(\frac12 - \sigma^{1/2}\bar\rho f_c \rho_c^{-1}\rt)\int_{\T^N} |u-v|^2 f\,dxdv\cr
&\quad + C_1\sigma\int_{\T^N} (\rho- \rho_c)^2 dx - C_2\sigma^{1/2}\int_{\T^N} \rho|u - m_c|^2 dx,
\end{aligned}
\end{align*}
where $C_1$ and $C_2$ are given by
\[
C_1 := C\lt( C - E_0\rho_c^{-1} - (1 + \bar \rho)\sigma^{1/2} - \lt(\mu + \lambda + \|\rho_f\|_{L^\infty(\R_+;L^{N/2}(\T^N))}\rt)\sigma \rt),
\]
and
\[
C_2:= C\bar\rho \lt( E_0 \rho_c^{-1} + \|u\|_{L^\infty(\T^N \times \R_+)} + \sigma^{1/2} + 1\rt),
\]
respectively. Here the constant $C > 0$ is a positive constant independent of $\sigma$, $E_0$, and $t$. Note that 
constant $C_1$ can be positive if we choose $\sigma > 0$ and $E_0 > 0$ small enough. 

Furthermore, we estimate 
\[
\int_{\T^N} \rho|u - m_c|^2dx \leq C\bar\rho\int_{\T^N}|\nabla u|^2 dx \quad \mbox{for some} \quad C > 0,
\]
by using the similar argument to \eqref{est_3_4}. This yields
\begin{align*}
\begin{aligned}
\mathcal{D}^\sigma(f,\rho,u) &\geq \lt(\frac{\mu}{2} - C C_2\bar\rho \sigma^{1/2}\rt)\int_{\T^N} |\nabla u|^2 dx + \frac{(\mu + \lambda)}{2}\int_{\T^N} |\nabla \cdot u|^2 dx \cr
&\quad + \lt(\frac12 - \sigma^{1/2}\bar\rho f_c \rho_c^{-1}\rt)\int_{\T^N \times \R^N} |u-v|^2 f\,dxdv+C_1\sigma\int_{\T^N} (\rho- \rho_c)^2 dx,
\end{aligned}
\end{align*}
and then we take $\sigma > 0$ and $E_0 > 0$ small enough to have
\[
\frac{\mu}{2} - C C_2\bar\rho \sigma^{1/2} > 0, \quad \frac12 - \sigma^{1/2}\bar\rho f_c \rho_c^{-1} > 0, \quad \mbox{and} \quad C_1 > 0.
\]
This completes the proof.
\end{proof}

\begin{proof}[Proof of Theorem \ref{thm_main}] 
Combining the estimates in Lemmas \ref{lem_cri0} and \ref{lem_cri}, we obtain that there exists a positive constant $C > 0$ such that
\begin{align*}
\begin{aligned}
\ml(f,\rho,u) &= \int_{\T^N \times \R^N} |v - v_c|^2 f\,dxdv + \int_{\T^N} \rho|u - m_c|^2dx + \int_{\T^N} (\rho - \rho_c)^2dx + |v_c - m_c|^2\cr
&\leq C\lt(\int_{\T^N} |\nabla u|^2 dx + \int_{\T^N} |\nabla \cdot u|^2 dx + \int_{\T^N \times \R^N} |u-v|^2 fdx\rt) + \int_{\T^N} (\rho - \rho_c)^2dx\cr
&\leq C\mathcal{D}^\sigma (f,\rho,u),
\end{aligned}
\end{align*}
for sufficiently small $\sigma > 0$ and $E_0 > 0$. Then by using this inequality and together with \eqref{equiv2}, we have
\[
\frac{d}{dt}\me^\sigma(f,\rho,u) + C\me^\sigma(f,\rho,u) \leq 0 \quad \mbox{for some} \quad C>0,
\]
and this concludes 
\[
\ml(f,\rho,u) \leq C\me^\sigma(f,\rho,u) \leq C\me^\sigma(f_0,\rho_0,u_0) e^{-Ct} \leq C\ml(f_0,\rho_0,u_0)e^{-Ct}, \quad t\geq 0,
\]
where $C >0$ is a positive constant independent of $t$.
\end{proof}

\section*{Acknowledgments}
{\small The author warmly thanks to the anonymous referee for helpful comments. The author was supported by Engineering and Physical Sciences Research Council(EP/K008404/1) and the ERC-Starting grant HDSPCONTR ``High-Dimensional Sparse Optimal Control''. The author is also supported by the Alexander von Humboldt Foundation.}

\end{document}